\documentclass{article}
\usepackage[utf8]{inputenc}
\usepackage{amsmath,amsfonts,amssymb}
\usepackage{a4wide,amsthm}
\usepackage[english]{babel}
\usepackage[babel=true]{csquotes}
\usepackage{textcomp}
\usepackage{comment}
\usepackage{hyperref} 
\usepackage{pdfpages}
\usepackage{pgf,tikz}
\usepackage{enumitem}
\usepackage{mathrsfs}
\usetikzlibrary{arrows}
\usepackage{mathrsfs}
\renewcommand{\div}{\operatorname{div}}
{\newtheorem{thm}{Theorem}[section]}
{\newtheorem{prop}[thm]{Proposition}}
{}
{\newtheorem{lemme}[thm]{Lemma}}
{}
{}

\providecommand{\R}{\mathbb{R}}

\providecommand{\N}{\mathbb{N}}

\renewcommand{\leq}{\leqslant}
\renewcommand{\geq}{\geqslant}

\renewcommand{\ge}{\geqslant}

\renewcommand{\div}{\operatorname{div}}
\newcommand{\curl}{\operatorname{curl}}
\newcommand{\trace}{\operatorname{tr}}
\newcommand{\as}{\operatorname{as}}

\newcommand{\bq}{\begin{equation}}
\newcommand{\eq}{\end{equation}}
\def\P{\mathbb P}

\newcommand{\Addresses}{{
  \bigskip
  \footnotesize

 \bigskip
 
\noindent  A.~Mecherbet, \textsc{Institut de Math\'ematiques de Jussieu-Paris Rive Gauche, Universit\'e Paris Cit\'e,
    8~Place Aur\'elie Nemours,
F75205 Paris Cedex~13, France}\par\nopagebreak
\noindent  \textit{E-mail address: }\texttt{mecherbet@imj-prg.fr}

  \bigskip

\noindent  F.~Sueur, \textsc{Department of Mathematics, Maison du nombre, 6 avenue de la Fonte, 
University of Luxembourg, L-4364 Esch-sur-Alzette, Luxembourg}\par\nopagebreak 
\noindent  \textit{E-mail address:} \texttt{Franck.Sueur@uni.lu}
}}

\title{On the analyticity of the flow map for 
 the AHT equations}

\author{Amina Mecherbet and Franck Sueur}

\date{\today}

\begin{document}

\maketitle

\abstract{The AHT equation is a non linear and non local vectorial  transport equation which was introduced in 2003 by Angenent, Haker and Tannenbaum in optimal transport theory. 
For this equation, classical solutions are known to exist at least locally in time, and a flow map can thus be uniquely associated with these solutions. In this paper we consider the case where the equation is set in a bounded domain with an analytic boundary and we prove that the flow map is analytic with respect to time. 
}

\tableofcontents

\maketitle

\section{Introduction and statement of the result}

Let $\Omega$  a smooth\footnote{this assumption will be reinforced later on.} bounded domain in $\R^d$, where $d\ge 2$. In \cite{AHT} 
 Angenent, Haker, and Tannenbaum   proposed the following equation, so-called the  AHT equation: 
\begin{equation} \label{AHT}\tag{AHT}
\begin{aligned}
\partial_t y+u\cdot \nabla y&=0,\\
u&=\P y,
\end{aligned}
\end{equation}
where $y = y(x,t) \in \R^d$, $x\in \Omega \subset \R^d$, $t\ge 0$, and  $\P$ denotes the classical Leray projector onto the space of divergence-free vector fields. 
  The Leray projector $u = \P y$ is defined as follows: for a given map $y:\Omega\to \R^d$, we construct the potential $p$ as the unique solution with zero mean value of the following boundary value problem:  
\begin{equation}\label{def-P} 
\left\{\begin{aligned}\Delta p &= \div  y  \qquad \mbox{in}\quad \Omega 
\\
\frac{\partial p}{\partial n} & = y \cdot n  \qquad \mbox{on}\quad \partial \Omega\end{aligned}
\right.\end{equation}
where $n$ is the unit outward normal to $\partial \Omega$. Then we define 
\begin{equation}\label{def_Leray_projection}
\P y = y - \nabla p.
\end{equation}
 As a consequence of the definition, the velocity $u=\P y$ is tangent to the boundary: 
\begin{equation}\label{BCs} 
u\cdot n = 0  \qquad \mbox{on} \quad \partial\Omega ,
\end{equation}  
and divergence free in $\Omega$: 
\begin{equation}\label{Bdiv} 
\div u = 0  \qquad \mbox{in} \quad \Omega .
\end{equation}  
The Leray projector is linear continuous in a wide range of Sobolev and  H\"older spaces. For instance there exists a constant $C_\Omega > 0$ such that for any $y$ in $C^{1,\alpha}(\Omega)$, 
\begin{equation}\label{eq:def_C_Omega}
\| \P y\|\leq C_\Omega \|y \|.
\end{equation}
where we use the following notations for $\alpha \in(0,1)$
\begin{eqnarray*}
    \| \cdot\|=\|\cdot\|_{C^{1,\alpha}(\Omega)}, 
    , &
       | \cdot|=\|\cdot\|_{C^{0,\alpha}(\Omega)} .
\end{eqnarray*}
We will also use the following H\"older space on the boundary.
$$  \| \cdot\|_{\partial \Omega} =\|\cdot\|_{C^{1,\alpha}(\partial \Omega)} .$$ 

Angenent, Haker and Tannenbaum have introduced this equation in the context of the optimal transport theory. 
We also refer to the paper \cite{Brenier09} by Brenier where the AHT equation is further discussed, again in the context of the optimal transport theory, and also where another interpretation related to the convection theory is given. 

Local-in-time existence and uniqueness of classical solutions of \eqref{AHT} 
is then easy to obtain by standard technics from nonlinear hyperbolic theory. 
Next result, stated in terms of  H\"older spaces $C^{1, \alpha}(\Omega)$ has been obtained in \cite{AHT}. As mentioned in \cite{Nguyen^2} it can be also stated alternatively  in some appropriate  Sobolev spaces.
\begin{thm}\label{lwp}
Let $\Omega$ be a bounded domain in $\R^d$, with smooth boundary. Let $\alpha$ in $(0,1)$. 
Then for any initial data $y_0\in C^{1, \alpha}(\Omega)$, there exist a positive time $T^\star>0$ depending only on $\| y_0\|_{C^{1, \alpha}(\Omega)}$ and a unique solution $y\in C((-T^\star, T^\star); C^{1, \alpha}(\Omega))$ of \eqref{AHT}.
\end{thm}

 Note that by  \eqref{def-P} the velocity 
 $u=\mathbb{P}y$ lies also in $C((-T^\star,T^\star),C^{1,\alpha})$.
 It is therefore classical to define, in a unique way, a corresponding flow map. 

\begin{prop}\label{flow}
Let $\Omega$ be a bounded domain in $\R^d$, with smooth boundary, $\alpha$ in $(0,1)$ and $y_0\in C^{1, \alpha}(\Omega)$
Let $y$ be the associated solution by Theorem \ref{lwp}. 
Then there is a unique associated flow  $\Phi$ in $C((-T^\star, T^\star); C^{1, \alpha}(\Omega))$ such that for any $x$ in $\Omega$, 
\begin{equation} \label{def-flot}
\partial_t  \Phi(t,x)  = u (t, \Phi(t,x) )  \text{ and }
\Phi(0,x) = x .
\end{equation}
\end{prop}
%

%

Our main result is that the flow map is actually analytic in time, if one assumes that the boundary of $\Omega$ is analytic. 
\begin{thm}\label{main}
Let $\Omega$ be a multiply connected analytic bounded domain in $\R^d$, $\alpha$ in $(0,1)$.
Let an initial data $y_0\in C^{1, \alpha}(\Omega)$, and the maximal positive time $T^\star>0$ be such that there is a unique solution $y\in C((-T^\star, T^\star); C^{1, \alpha}(\Omega))$ of \eqref{AHT}.
Let $\Phi$  be the corresponding flow map. 
Then $\Phi$ is analytic from $(-T, T)$ to $C^{1, \alpha}(\Omega)$ for some $T\leq T^\star$ depending on $y_0$ and $\Omega$. 
\end{thm}
Theorem \ref{main} extends to the case of the AHT equations a result which was already known for the incompressible Euler equation. 
We refer to \cite{Chemin92,Chemin, gamblinana, gamblinana2} ; and the many extensions; for results proving smoothness, analyticity or Gevrey regularity of the flow map and in particular to the case where the system is contained in a bounded domain: \cite{Kato, ogfstt}. Further results on the analyticity of the trajectories for several incompressible fluid models can be found in \cite{CVW,z,hernandez,inci,besse2020regularity}.  We refer also to Section \ref{section2} for more details on the Kato method.

The AHT equation is reminiscent of the 2D Euler equation in vorticity form: the transporting vector field $u$ is obtained by a pseudo-differential operator from the transported quantity. However, in the case of the AHT equation, this pseudo-differential operator is of order $-1$ while it is of order $0$ for the Euler equation. Furthermore, the transported quantity is a vector field rather than a scalar field. 

 The AHT system is also similar to the inviscid Burgers equation with an extra non-local feature, due to the Leray projection. 
 Since it is pseudo-differential operator of order zero, it is also reminiscent of the incompressible porous medium equation, also called IPM equation. More precisely, for any $\rho$ that satisfies the 2D IPM equation, the vector $y=-(0,\rho)$ is a solution to the AHT system in the 2D setting. 

\paragraph{Organization of the paper}
The paper is organized as follows. In Section 2 we present the main idea introduced by Kato in \cite{Kato}, allowing to prove the analyticity of the trajectories in the context of the Euler equation. We define in particular the main notations with regard to the $\div$, $\curl$, as well as commutator rules and recursive formulas for the $\div$ and the boundary estimates.
In Section 3 we explain the main strategy for proving Theorem \ref{main} by adapting the Kato method to the AHT system. Sections 4 and 5 are devoted to proving recursive formulas for the $\curl$ and circulations, respectively. In Section 6 we gather all the arguments to prove the main theorem.
Eventually, in Section \ref{section7} we discuss some remarks regarding the analyticity in time as well as open questions related to the AHT system.

\section{The Kato method for proving smoothness of the trajectories}\label{section2}

In \cite{Kato}, Kato proved smoothness in time of the trajectories of the flow associated to the solution of Euler equations 
\begin{equation} \label{euler}
\begin{aligned}
\partial_t u+u\cdot \nabla u + \nabla p&=0, \text{ on } \Omega\\
\div u &=0,\text{ on } \Omega\\
u\cdot n &=0, \text{ on } \partial \Omega
\end{aligned}
\end{equation}
 on a domain $\Omega$ corresponding either to $\R^m$ or to a bounded domain with smooth boundaries.
 The main idea is the observation that, by the definition of the flow in \eqref{def-flot} and the chain rule, one has for any $k$ in $\N$, 
\begin{equation} \label{chainrule}
\partial^{k +1 }_t \Phi(t,x)= D^k u (t, \Phi (t,x)),
\end{equation}
where $D$ denotes the material derivative
\begin{equation*}
D := \partial_{t} + u \cdot \nabla .
\end{equation*}
Hence, the problem aims at proving smoothness of the successive material derivatives of the velocity $ D^k u$. The main strategy used by Kato can be summarized in a two-step method.
\begin{itemize}
\item Compute recursive formulas for $\div D^k u$ and $\curl D^k u$ in terms of $\nabla D^p u$ with $p < k$, assuming inductively that the latter exist.
\item Deduce an estimate on $D^k u$ itself using elliptic regularity lemmas involving $\div D^k u$ and $\curl D^k u$.
\end{itemize} 
As explained by Kato, this method can be adapted to the case of a bounded domain by considering further information on the regularity of $D^ku$ on the boundary. See Lemma \ref{Lemme1} for more details. 

Still in the case of Euler equation and motivated by Gamblin's results \cite{gamblinana,gamblinana2}, the combinatorics in Kato approach have been refined in \cite{ogfstt} to obtain the analyticity of the flow by proving recursive estimates on $D^k u$. 
Let us also refer to  \cite{Sueur} for a similar result regarding the Gevrey regularity
of  Yudovich type solutions  to the $2$D incompressible Euler equations in a bounded domain. 

We  describe  this strategy below. 

\subsection{Commutator rules and main notations.} First, we perform the computations for $k=1$, one has  
\begin{equation}\label{dbase} 
\left\{\begin{aligned}
\div Du &=   \trace\left\{\nabla u  \,\nabla u \right\}  ,
\\ \curl Du &= 0 .
\end{aligned}
\right.\end{equation}
where we used equation \eqref{euler} for the second line and only the fact that $u$ is divergence free for the first line.
Above and throughout the paper, we use the convention that the curl of a vector-valued function in $\R^d$ is the skew-symmetric part of its gradient, that is, a $d\times d $ square matrix rather than a vector. 

Observe in particular that the operator $D$ has disappeared from the right-hand side, and if we consider the case where the equation is set in the full space with appropriate decay at infinity then one can deduce that $Du$ has the same regularity as $u$, since the derivative which is consumed in the r.h.s. of the divergence equation is compensated by the gain of one derivative in solving the div/curl system. 

This analysis can be continued by considering, iteratively, the higher-order derivatives $D^k$. This relies on the following commutator rules to exchange $D$ and differentiations with respect to $x$. The first is valid for any scalar field  $\psi$ defined in the fluid domain:
 \begin{eqnarray}
	\nabla (D\psi) -D(\nabla \psi) = (\nabla u)^\top \cdot (\nabla \psi), \label{t3.1}
 \end{eqnarray}
 If $\psi$is a  vector field defined in the fluid domain:
 \begin{eqnarray}
 \nabla (D\psi) -D(\nabla \psi) =(\nabla \psi)\cdot (\nabla u), \label{t3.2}\\
	\div D\psi - D\div \psi = \trace \left\{(\nabla \psi)\cdot (\nabla u) \right\}, \label{t3.3}\\
	\curl D\psi - D\curl \psi = \as \left\{(\nabla \psi)\cdot (\nabla u)  \right\}. \label{t3.4}
\end{eqnarray}
here $\trace\{A\} $ is the trace 
of $A \in {\mathcal M}_{3}(\R)$ and  $\as\{A\}:=A-A^*$ the antisymmetric part of $A \in {\mathcal M}_{3}(\R)$. We recall that the curl is defined as a square matrix rather than a vector.

We introduce for $s \in \N^*$ and $ \alpha := ( \alpha_1,\ldots, \alpha_s )  \in \N^s$ the notations $$ | \alpha | : = \alpha _1 + \ldots + \alpha _s, \quad \alpha! : = \alpha _1! \ldots \alpha _s!$$
We recall the following formal identity for divergence operator that can be found in \cite[Proposition 2]{ogfstt}.
\begin{prop}\label{P1}
For $k \in \N^*$, for any smooth vector divergence free vector field $u$ in $\overline \Omega$
\begin{eqnarray}\label{P1f}
\div D^k u=      \trace\left\{
\underset{s=2}{\overset{k+1}{\sum}} \underset{\underset{|\alpha|=k+1-s}{\alpha \in \N^s }}{\sum}
c^1_k (s,\alpha  )  \,  D^{\alpha_1}\nabla  u \cdot \ldots  \cdot D^{\alpha_s} \nabla u \right\} 
\end{eqnarray}
where  $c^1_k (s,\alpha )$ are integers satisfying
\begin{equation} \label{Ci:1}
|c^1_k (s,\alpha) | \leqslant \frac{k ! }{\alpha ! } .
\end{equation}
\end{prop}
Observe that the above proposition only assumes that the vector field $u$ is divergence-free, while the first equation of \eqref{euler} is not required. 
Moreover, the r.h.s. of \eqref{P1f} only involves some iterated derivatives $D^j u$ with $j < k$. 

A similar identity can be obtained for $\curl D^k u$ when $u$ is a solution of 
 both equations in \eqref{euler}.
\subsection{Case of a bounded domain (regularity lemma)}
In the case where the system is set on a bounded domain, the inversion of the div/curl system requires some boundary conditions. 
Indeed we recall the following regularity result, see \cite[Lemma 1]{ogfstt} 
 where $(\Gamma_i)_i$, for  $(i=1,\ldots,g)$, is a family of smooth oriented loops which generates a basis of the first singular homology space of $\Omega$ with real coefficients. These are  in particular  some smooth closed curves in $\Omega$. Let $\Pi$ be the mapping defined by 
$$u\mapsto \left(\oint_{\Gamma_1} u\cdot \tau d\sigma,\ldots, \oint_{\Gamma_g} u\cdot \tau d\sigma\right),$$
with $\tau$ denoting the tangent vector.
\begin{lemme}[Regularity]\label{Lemme1}
There exists a constant $c_{\mathfrak r}$ depending only on $\Omega$ and $\Gamma_i$ for $1\leq i\leq g$ such that for any $u \in C^{0,\alpha}(\Omega)$ such that
$$
\div  u\in C^{0,\alpha}(\Omega), \quad \curl  u \in C^{0,\alpha}(\Omega), \quad  u\cdot n \in C^{1,\alpha}(\partial \Omega),
$$
one has $ u\in C^{1,\alpha}(\Omega)$ and 
\begin{equation}
\label{reg1}
\| u\| \leq c_{\mathfrak r} \left(|\div  u| + |\curl  u| + \| u\cdot n\|_{\partial \Omega} + |\Pi  u|_{\R^{g}} \right).
\end{equation}
\end{lemme}

The idea is then to apply Lemma \ref{Lemme1} to $D^{k} u$ to get 
\begin{equation} \label{fond}
\| D^{k} u \| \leq c_{\mathfrak r} \left(| \div D^k u | + | \curl D^k u| +  \| D^{k} u\cdot n \|_{\partial \Omega}+ | \Pi D^k u |_{\R^{g}} \right) .
\end{equation}
Hence one needs to estimate each of the four terms appearing in the above right hand side. 
The assumption on the analyticity of the boundary shows up to handle the term 
$D^{k} u\cdot n $. 
Indeed, if the boundary $\partial \Omega$ is analytic, then there is   a function $\rho_{\Omega}$
 defined on a neighborhood of $\partial \Omega$ which is equal to  the
signed distance to $\partial \Omega$, negative inside $\Omega$, in a smaller neighborhood of $\partial \Omega$ and such that  there exists
$c_\rho > 1$ such  that for all $s \in \mathbb{N}$, 
\begin{equation} \label{RhoAnalytique}
\|\nabla^s \rho_{\Omega} \| \leq c_{\rho}^s \, s! .
\end{equation}
Moreover there are some recursive  identities to obtain the terms 
$D^{k} u\cdot n $, if one assumes that $u$ is a  smooth tangent  vector field  in $\overline \Omega$, see  \cite[Proposition 2]{ogfstt}.
\begin{prop}\label{P1norm}
For $k \in \N^*$, for any 
smooth tangent  vector field  in $\overline \Omega$, we have 
on the boundary $ \partial \Omega$:
\begin{eqnarray}\label{P4f}
n\cdot D^k u = \underset{s=2}{\overset{k+1}{\sum}} \underset{\underset{|\alpha|=k+1-s}{\alpha \in \N^s }}{\sum}  c^2_k (s,\alpha )  \,\nabla^s \rho_{\Omega} \{ D^{\alpha_1} u , \ldots  ,D^{\alpha_s} u \} ,
\end{eqnarray}
where  the $ c^2_k (s,\alpha)  $ are negative integers satisfying
\begin{equation} \label{Ci:2}
| c^2_k (s,\alpha)  | \leqslant \frac{k ! }{\alpha ! (s-1)!}.
\end{equation}
\end{prop}
In the case where the vector field $u$ satisfies Euler equation \eqref{euler}, there exists in addition a formula for the circulations $\Pi D^k u$ which allows to conclude using the  regularity lemma \ref{Lemme1} combined with a recursive argument. Let us explain in the next section how to adapt such an argument for the AHT system.

\section{Scheme of proof of Theorem \ref{main}: Kato's method for the AHT system.}

According to the previous Section, we emphasize that Kato's strategy is not intrinsic to Euler equation and can be adapted to any context provided that one can extract suitable formulas for the $\div$, $ \curl$ operators as well as for the boundary terms. 

For the equation \eqref{AHT}, it may look unclear that considering a div/curl system is relevant.
However, the velocity $u$ associated with a solution to \eqref{AHT} is also divergence-free, and this is all that we need to apply \eqref{P1f} 
as we stressed above. Similarly, such an observation holds for the boundary term since $u$ is also tangent to the boundary, see Proposition \ref{P1norm} where no information coming from the equation is needed. Let us describe how to handle the two remaining terms: rotational and circulations using the \eqref{AHT} equation.

\paragraph{Rotational.}
Using \eqref{t3.4} we have
    $$
    \curl D u = D \curl u +\as \left\{(\nabla u)\cdot (\nabla u)  \right\} .
    $$
  On the other hand, thanks to \eqref{def_Leray_projection},  observe that 
  $$\curl u = \curl y ,$$
  hence, applying the curl operator to the AHT equation \eqref{AHT} together with \eqref{t3.4}, one gets
$$
D \curl y= - \as  \left\{(\nabla y)\cdot (\nabla u)  \right\} .
$$
 Combining the three last equalities we arrive at 
 \begin{equation}\label{ibis}
    \curl Du= -as \left\{ \nabla y\cdot \nabla u \right\} +as \left\{ \nabla u \cdot \nabla u \right\} . 
\end{equation}
Therefore, there is no loss of derivative from $y$ and $u$ to $Du$.

\paragraph{Circulations.}
  First we observe that  
\begin{equation} \label{sam1}
    Du =  D \mathbb{P}y =  \mathbb{P} Dy - [\mathbb{P}, u \cdot \nabla]y = - [\mathbb{P},u\cdot \nabla] y ,
\end{equation} 
since $ Dy = 0$  by  \eqref{AHT}. 
Moreover, by definition of the Leray projector, we have 
\begin{equation}\label{dec-1}
y =u+\nabla p,
\end{equation}
where $u =\P y$ and $p$ solves
\begin{equation}\label{cmt:elliptic11}
\begin{cases}
\Delta p= \div y\quad\text{in}~\Omega,\\
\frac{\partial p}{\partial n}=y\cdot n\quad\text{on}~\partial\Omega.
\end{cases}
\end{equation}
Similarly, we write
\begin{equation}\label{dec-2}
(u\cdot \nabla)y=\P((u\cdot \nabla)y)+\nabla g,
\end{equation}
with $g$ is the unique solution to the following Neumann problem:
\begin{equation}\label{cmt:elliptic2}
\begin{cases}
\Delta g= \div((u\cdot \nabla)y) \quad\text{in}~\Omega,\\
\frac{\partial g}{\partial n}=(u\cdot \nabla)y\cdot n\quad\text{on}~\partial\Omega.
\end{cases}
\end{equation}
%

Combining \eqref{dec-1} and \eqref{dec-2}, we have
\begin{equation}\label{comm}\begin{aligned}
~[\P, u\cdot\nabla]y&=\P((u\cdot \nabla) y)-(u\cdot\nabla)\P y\\
&=(u\cdot \nabla)y-\nabla g-(u\cdot \nabla)(y- \nabla p)\\
&=(u\cdot \nabla)(\nabla p)-\nabla g \\
&=\nabla(u\cdot \nabla p-g)- (\nabla u)^\top \nabla p \\
&=\nabla[u\cdot (y-u) -g]- (\nabla u)^\top (y-u) . \\
\end{aligned}
\end{equation}
Combining \eqref{sam1} and \eqref{comm}, we arrive at 
\begin{equation}\label{eq:Du}
    Du=  \nabla P + (\nabla u)^\top (y-u) ,
    \end{equation}
where
\begin{equation} \label{def-g}
    P =g-u\cdot (y-u) .
\end{equation}

\section{Recursive formulas for the rotational}\label{section3}

We now turn to the recursive formulas for the $\curl$. 
\begin{prop}\label{P2}
For any smooth solution $y$
of \eqref{AHT} in $\overline \Omega$,  for $k\geq 1$
\begin{equation}\label{eq:formule_rec_curl}  \curl D^k u  =\as\left(  \sum_{s=1}^k \sum_{\underset{| \alpha |    = k - s}{\alpha \in \N^s}} c^3_{k} (s,\alpha )  \, \nabla y\cdot \nabla D^{\alpha_1} u \cdot \ldots  \cdot \nabla D^{\alpha_s} u   
+
\underset{s=2}{\overset{k+1}{\sum}} \underset{\underset{|\alpha|=k+1-s}{\alpha \in \N^s }}{\sum}  c_{k}^4(s,\alpha)  \, \nabla D^{\alpha_1} u \cdot \ldots  \cdot \nabla D^{\alpha_s} u
\right\}
\end{equation}
with
$$
|c^3_{k} (s,\alpha )| \leq \frac{k!}{\alpha!},\quad |c^4_{k} (s,\alpha  )| \leq \frac{k!}{\alpha!} .
$$
\end{prop}
\begin{proof}
    For $k=1 $, it is a consequence of \eqref{ibis}.
    We argue then by induction for both $c_{k+1}^3$ and $c_{k+1}^4$. Assume that \eqref{eq:formule_rec_curl} is valid up to $k$. We have using \eqref{t3.4}
    \begin{align}\label{eq:calcul_rec_curl}
         \curl D^{k+1} u&= \curl D (D^k u) \notag\\
         &= D(\curl D^k u )+ \as [\nabla D^k u \cdot \nabla u ] \notag \\
         &=D \as\left\{  \sum_{s=1}^k \sum_{| \alpha |    = k - s} c^3_{k} (s,\alpha )  \, \nabla y\cdot \nabla D^{\alpha_1} u \cdot \ldots  \cdot \nabla D^{\alpha_s} u   
+
\underset{s=2}{\overset{k+1}{\sum}} \underset{\underset{|\alpha|=k+1-s}{\alpha \in \N^s }}{\sum}  c_{k}^4(s,\alpha  )  \, \nabla D^{\alpha_1} u \cdot \ldots  \cdot \nabla D^{\alpha_s} u
\right\} \notag \\
&+ \as [\nabla D^k u \cdot \nabla u ] .
    \end{align}

        \paragraph{Estimate of $c_{k+1}^3 (s,\alpha)$.} 
  In \eqref{eq:calcul_rec_curl}, when applying the material derivative $D$ to a term of type $\nabla y\cdot \nabla D^{\alpha_1} u \cdot \ldots  \cdot \nabla D^{\alpha_s} u $ and using \eqref{t3.2} we get two possible terms:
    \begin{enumerate}[label=\roman*)]
        \item either $\nabla y\cdot \nabla D^{\alpha_1} u \cdot \ldots  \nabla D^{\alpha_i+1} u\ldots  \cdot \nabla D^{\alpha_s} u $  for some $i\in \{1,\cdots,s\}$ 
        \item  or $\nabla y\cdot \nabla D^{\alpha_1} u \cdot \ldots  \nabla D^{\alpha_i} u \cdot \nabla u\cdot \nabla D^{\alpha_{i+1}} u \ldots  \cdot \nabla D^{\alpha_s} u $ for some $i\in \{0,\cdots,s-1\}$ with the convention that $i=0$ corresponds to a term of type $ \nabla y\cdot \nabla  u  \cdot \nabla D^{\alpha_{1}} u \ldots  \cdot \nabla D^{\alpha_s} u$.
    \end{enumerate}

Let $s\in \{1, \cdots ,k\}$ and $\alpha\in \N^s$,$|\alpha|=k+1-s$. The coefficient $c_{k+1}^3 (s,\alpha)$ associated to a term $\nabla y \cdot \nabla D^{\alpha_1}\cdots \nabla D^{\alpha_s} $ would be obtained:
\begin{itemize}
    \item either from i) with coefficient $c^3_{k}(s,\alpha-e_i) $  for some $i=1, \cdots , s$ where $$e_i=(0,0,\cdots,1,0,\cdots,0)\in \N^s ,$$ 
\item either from ii) if $\alpha \in \N^s$ is of the form 
$$\alpha=(\alpha_1,\cdots, \alpha_{i_0-1},0,\alpha_{i_0+1}, \cdots,\alpha_s), \quad i_0=1,\cdots , s-1 ,$$
and in this case the coefficient would be $c^3_{k}(s-1,\beta)$ with 
$$
\beta=(\alpha_1,\cdots, \alpha_{i_0-1},\alpha_{i_0+1}, \cdots,\alpha_s)\in \N^{s-1} \text{ with } |\beta|=|\alpha|.
$$
 Note that if $i_0=1$, then $\alpha$ has to be understood as
$
(0,\alpha_2,\cdots,\alpha_{s})
$
and the corresponding $\beta$ is $\beta=(\alpha_2,\cdots,\alpha_{s})\in \N^{s-1}$.
Hence, in general, for $s=1 \cdots k$ we have
\begin{align*}
|c^3_{k+1}(s,\alpha)|& \leq \left(\underset{i=1}{\overset{s}{\sum}} |c_k^3(s,\alpha-e_i)|\right)+ |c_k^3\left(s-1,(\alpha_1,\cdots,\alpha_{i_0-1},\alpha_{i_1},\cdots,\alpha_{s})\right)|\\
&\leq \left(\underset{i=1}{\overset{s}{\sum}} \frac{k!}{(\alpha-e_i)!}\right) + \frac{k!}{\alpha!}\\
&=\frac{k!}{\alpha!}\left(\underset{i=1}{\overset{s}{\sum}}\alpha_i\right)+\frac{k!}{\alpha!}= \frac{k!}{\alpha!}[(k+1-s) + 1]\leq \frac{(k+1)!}{\alpha!} .
\end{align*}
Finally if $s=k+1$, then $\alpha=0\in \N^{k+1}$ and we are considering the coefficient in front of 
$$\underbrace{\nabla y \cdot \nabla u \cdots \nabla u}_{k+2 \text{ terms}} ,$$
which can be obtained only from ii) when adding a coefficient $\nabla u$ to one of the terms in $\underbrace{\nabla y \cdot \nabla u \cdots \nabla u}_{k+1 \text{ terms}} $ with coefficient $c^3_k(k,0)$  hence 
$$
|c_{k+1}^3(k+1,0)| \leq  \underset{i=1}{\overset{k+1}{\sum}}|c_{k}^3(k,0))|\leq (k+1) \frac{k!}{0!}=(k+1)! ,
$$
which yields the desired result in the case $(s,\alpha)=(k+1,0)$.

\paragraph{Estimate of $c_{k+1}^4 (s,\alpha)$.}
The proof is analogous, the only difference lies in the case where $s=2$ and $\alpha \in \N^2$, $|\alpha|=k$ . In this case there is no term coming from i), the terms coming from ii) come from terms of type $\nabla D ^{\beta_1}u \cdot \nabla D^{\beta2} u$ with $\beta=(\beta_1,\beta_2)=\alpha-e_i$ and such that $|\beta|= k-1$.
However, one has to take into account another possibility  that is the last term appearing in \eqref{eq:calcul_rec_curl} with coefficient 1 hence 
\begin{align*}
|c_{k+1}^4(2,\alpha)|& \leq \left(\underset{i=1}{\overset{2}{\sum}} |c_k^4(2,\alpha-e_i )|\right)+ 1\\
&\leq \left(\underset{i=1}{\overset{2}{\sum}} \frac{k!}{(\alpha-e_i)!}\right) +1\\
&=\frac{k!}{\alpha!}\left(|\alpha|\right)+1= \frac{k!}{\alpha!}k+1= \frac{k! k + \alpha!}{\alpha!}\leq  \frac{k! k + k!}{\alpha!}=\frac{(k+1)!}{\alpha!},
\end{align*}
where we used that $\alpha!=\alpha_1!\cdot \alpha_2! \leq |\alpha|!=k! $ since $\alpha=(\alpha_1,\alpha_2)=(|\alpha|-p,p)=(k-p,p)$ for some $1\leq p\leq k-1$. 
\end{itemize}
\end{proof}

\section{Recursive formulas for the circulations}\label{section3bis}

We now turn to the recursive formulas for the circulations. 
\begin{prop}\label{P3}
\begin{equation} \label{eq:formule_rec_circ}
D^{k} u =   \nabla D^{k-1} P  + K^{k} [u,y-u]  ,
\end{equation}
where for $k\geq 1$,
\begin{equation}\label{eq:formule_rec_circ2}
K^{k}[u,\psi]=\sum_{r=1}^{k}  c_{k,r} \nabla \left(D^{ r-1} u\right) ^\top \cdot D^{k-r} \psi, \quad |c_{k,r}|\leq \dbinom{k}{ r},
\end{equation}
where $P$ is given by \eqref{def-g}.
\end{prop}
\begin{proof}
 The recursive formula for $k=1$ is given by \eqref{eq:Du}, $c_{1,1}= 1$. 
Let us assume that it holds up to order $k$. 
    We then apply the material derivative $D$ to 
    \eqref{eq:formule_rec_circ}, use \eqref{eq:formule_rec_circ2} and the Leibnitz' rule to obtain that 
    \begin{align} \label{ici}
        D^{k+1}u&=D\nabla D^{k-1} P +\sum_{r=1}^{k}  c_{k,r} D\nabla \left(D^{ r-1} u\right)^\top \cdot D^{k-r}  (y-u)
        \\ \notag   &\quad +\sum_{r=1}^{k}  c_{k,r} \nabla \left(D^{ r-1} u\right) ^\top \cdot D^{k+1-r} (y-u) . 
    \end{align}
 Let us now mention the following formula 
    $$
    D \nabla \psi^\top=\nabla D \psi ^\top-\nabla u^\top \cdot \nabla \psi^\top ,
    $$
    which is true for any smooth vector field $\psi$.
     Using the commutation rule \eqref{t3.1} to deal with the first term in \eqref{ici} 
     and the previous formula with $\psi =D^{ r-1} u$ to deal with the first factor of the second term in  \eqref{ici}  we get 
  \begin{align}\notag 
        D^{k+1}u&=\nabla  D^{k} P  -\nabla u^\top \cdot \nabla D^{k-1} P +\sum_{r=1}^{k}  c_{k,r} \left(\nabla \left(D^{ r} u\right) ^\top \cdot D^{k-r}(y-u)- \nabla u^\top \cdot \nabla \left(D^{ r-1} u\right)^\top \cdot D^{k-r}  (y-u)\right) \\
      \label{la}    &\quad +\sum_{r=1}^{k}  c_{k,r} \nabla \left(D^{ r-1} u\right)^\top \cdot D^{k+1-r}  (y-u).
         \end{align} 
Now, we observe that the second term and the second half of the term coming from the first sum (the third term) 
of the r.h.s. above can be simplified by using  the recursive formula at order $k$, that is 
  \begin{align}  \notag 
       -\nabla u^\top \cdot \nabla D^{k-1} P - \sum_{r=1}^{k}  c_{k,r} \nabla u^\top \cdot \nabla \left(D^{ r-1} u\right)^\top \cdot D^{k-r}  (y-u)
        &= - \nabla u^\top \cdot  \left( \nabla D^{k-1} P  + K^{k} [u,y-u] \right)
 \\ \label{la2}  &= - \nabla u^\top \cdot  D^{k}u . 
         \end{align} 
Combining \eqref{la} and \eqref{la2} we arrive at 

  \begin{align*}
        D^{k+1}u  &= \nabla  D^{k} P - \nabla u^\top \cdot D^k u + \sum_{r=1}^{k}  c_{k,r} \left( \nabla \left(D^{ r} u\right) ^\top \cdot D^{k-r}  (y-u)+ \nabla \left(D^{ r-1} u\right) ^\top \cdot D^{k+1-r} (y-u)\right) .
     \end{align*} 
     This entails the following. 
     \begin{itemize}
         \item  For $r=1$, the coefficients $c_{k,1}$ and $-1$ (coming from the term $-\nabla u^\top \cdot D^k u$ ) contribute to $c_{k+1,1}$ hence
    $$
    |c_{k+1,1}|\leq |c_{k,1}|+1\leq k+1 = \dbinom{k+1}{ 1}.
    $$
 \item   For $r=2,\cdots, k$, only the coefficients $c_{k,r} $ and $c_{k,r-1} $ contribute to $c_{k+1,r}$, hence
    $$
    |c_{k+1,r}|\leq |c_{k,r}|+|c_{k,r-1}| \leq \dbinom{k}{ r}+\dbinom{k}{ r-1}= \dbinom{k+1}{ r} .
    $$
     \item  Finally for $r=k+1$, only the coefficient $c_{k,k}$ contributes to the coefficient $c_{k+1,k+1}$ hence
    $$
    |c_{k+1,k+1}|\leq |c_{k,k}| =1=\dbinom{k+1}{k+1} .
    $$
     \end{itemize}
    This proves    \eqref{eq:formule_rec_circ}-\eqref{eq:formule_rec_circ2} at order $k+1$ and concludes the proof of Proposition \ref{P3}. 
\end{proof}

\section{Proof of Theorem \ref{main}}
\label{section4}

Let $|\Gamma_i|$ denoting  the length of the loops involved in Lemma \ref{Lemme1}, we introduce the constant 
\begin{equation}\label{eq:def_C_gamma}
C_\Gamma:= \left(\underset{i=1}{\overset{g}{\sum}} |\Gamma_i|^2\right)^{1/2} .
\end{equation}
Recalling that the constant $C_\Omega$ is the one appearing in \eqref{eq:def_C_Omega}, let us also introduce for $L>0$, 
 \begin{equation}\label{def_gamma}
 \gamma(L)=\underset{k\geq 1}{\sup} \left \{ 
 4\underset{s=2}{\overset{k+1}{\sum}} s c_\rho^s L^{1-s} \frac{(k+1)^2 20^s}{(k+2-s)^2}+ C_\Gamma  \frac{C_\Omega +1}{C_\Omega L}\underset{r=1}{\overset{k}{\sum}} \frac{(k+1)^2}{r^3(k-r+1)^2}
 \right \} .
 \end{equation}
 We assume that $L$ is large enough such that $L>1/C_\Omega$ and 
\begin{equation}\label{hyp_L}
\gamma(L)\leq \frac{1}{c_{\mathfrak r}} ,
\end{equation}
where the constant $c_\mathfrak{r}$ appears in Lemma \ref{Lemme1}.

Recall that $u$ is the divergence free part of $y$ satisfying \eqref{AHT}.
We are going to prove by induction that, supposing that $u$ is smooth, there exists $t^\star \in (0,T)$ such that one has for all $k\in \mathbb{N}$ and all $t \in [-t^\star,t^\star]$, 
\begin{equation}
\label{indu}
\| D^k u(t) \|  \leq  C_\Omega \frac{k! (C_\Omega L)^k }{(k+1)^2} \|y(t)\|^{k+1} ,
\end{equation}
 Indeed, such an estimate proves the analyticity for smooth solutions. In the general case, the idea is to regularize the initial data $y_0$ to obtain a sequence of $\mathcal{C}^\infty$ solutions $(y^n)_n$ with a uniform maximal time existence in $n$ and smooth velocities $(u^n)_n$ due to the continuity of the Leray projector. Smoothness in time is recovered from smoothness in space variable thanks to the equations. Hence we recover the analyticity of the flows $\Phi^n$ using the estimate \eqref{indu} for each velocity $u^n$ with constants independent of $n$ and conclude by passing in the limit when $n\to \infty$.

\paragraph{Proof of \eqref{indu} by induction.}
For $k=0$, there is nothing to prove. 

Let us assume that \eqref{indu} holds up to order $k-1$.  The idea is to
apply Lemma \ref{Lemme1} to $D^{k} u=D^{k} \mathbb{P}y$ to get 
%
\begin{equation}
\label{indu2}
\| D^{k} u \| \leq c_{\mathfrak r} \left(| \div D^k u | + | \curl D^k u| +  \| D^{k} u\cdot n \|_{\partial \Omega}+ | \Pi D^k u |_{\R^{g}} \right) .
\end{equation}

We now recall \cite[Lemma 7.3.3]{Chemin}, which will be of repetitive use. 
%

%
%
\begin{lemme} \label{LemmeChemin}
For any couple of positive integers $(s,m)$ we have
\begin{equation}\label{DefUpsilon}
\sum_{\substack{{\alpha \in \N^{s}} \\ {|\alpha|=m} }}  \Upsilon(s,\alpha) \leq \frac{20^{s}}{(m+1)^{2}}, \text{ where }
\Upsilon(s,\alpha):=\prod_{i=1}^s \frac{1}{(1+\alpha_{i})^2}.
\end{equation}
\end{lemme}
Let us now estimate each of the four terms appearing on the right-hand side of \eqref{indu2}. 

\paragraph{Divergence.}
Using Proposition \ref{P1} together with \eqref{indu} we have 
\begin{align}\label{ineq_div}
| \div D^k u |& \leq \, \underset{s=2}{\overset{k+1}{\sum}}\,\underset{|\alpha|=k+1-s}{\sum} \frac{k!}{\alpha!}\,\underset{i=1}{\overset{s}{\prod}} \|D^{\alpha_i}u\|\\ \notag
&\leq  \, \underset{s=2}{\overset{k+1}{\sum}} \underset{|\alpha|=k+1-s}{\sum}\frac{k!}{\alpha!} \,\underset{i=1}{\overset{s}{\prod}}\frac{\alpha_i !C_\Omega (C_\Omega L)^{\alpha_i}}{(\alpha_i+1)^2} \|y \|^{\alpha_i+1} \\ \notag
&= k!C_\Omega (C_\Omega L)^k \|y\|^{k+1}  \, \underset{s=2}{\overset{k+1}{\sum}} L^{1-s}\underset{|\alpha|=k+1-s}{\sum} \,\underset{i=1}{\overset{s}{\prod}}  \frac{1}{(\alpha_i+1)^2}\\ \notag
&\leq C_\Omega \frac{k! (C_\Omega L)^k \|y\|^{k+1}}{(k+1)^2}  \, \underset{s=2}{\overset{k+1}{\sum}} L^{1-s}\,\frac{(k+1)^2 20^s}{(k+2-s)^2},
\end{align}
where we used Lemma \ref{LemmeChemin}. 

\paragraph{Rotational.} For the $\curl$,  using Proposition \ref{P2}, we have that 
\begin{align}\label{ineq_curl}
| \curl D^k u |& \leq \|y\|  \, \underset{s=1}{\overset{k}{\sum}} \, \underset{|\alpha|=k-s}{\sum}  \frac{k!}{\alpha!} \  \underset{i=1}{\overset{s}{\prod}}\|D^{\alpha_i}u\|+   \, \underset{s=2}{\overset{k+1}{\sum}} \, \underset{|\alpha|=k+1-s}{\sum}\  \frac{k!}{\alpha!} \underset{i=1}{\overset{s}{\prod}}\|D^{\alpha_i}u\| \\ \notag
&\leq \|y\|  \, \underset{s=1}{\overset{k}{\sum}}\underset{|\alpha|=k-s}{\sum} \frac{k!}{\alpha!}\  \underset{i=1}{\overset{s}{\prod}}\frac{\alpha_i !C_\Omega (C_\Omega L)^{\alpha_i}}{(\alpha_i+1)^2} \|y \|^{\alpha_i+1} +  \, \underset{s=2}{\overset{k+1}{\sum}} \, \underset{|\alpha|=k+1-s}{\sum} \frac{k!}{\alpha!}\   \underset{i=1}{\overset{s}{\prod}}\frac{\alpha_i !C_\Omega (C_\Omega L)^{\alpha_i}}{(\alpha_i+1)^2} \|y \|^{\alpha_i+1} \\ \notag
&= k! L^k \|y\|^{k+1} (C_\Omega L)^k \,\left(\underset{s=1}{\overset{k}{\sum}} L^{-s}\,\underset{|\alpha|=k-s}{\sum} \ \underset{i=1}{\overset{s}{\prod}}  \frac{1}{(\alpha_i+1)^2}+ \underset{s=2}{\overset{k+1}{\sum}}C_\Omega L^{1-s}\,\underset{|\alpha|=k+1-s}{\sum}\ \underset{i=1}{\overset{s}{\prod}}  \frac{1}{(\alpha_i+1)^2} \right) \\ \notag
&\leq  C_\Omega \frac{k! (C_\Omega L) ^k \|y\|^{k+1}}{(k+1)^2}  \,\left(\underset{s=1}{\overset{k}{\sum}} \frac{L^{-s}}{C_\Omega} \frac{(k+1)^2 20^s}{(k-s+1)^2}+ \underset{s=2}{\overset{k+1}{\sum}} L^{1-s}\,\frac{(k+1)^2 20^s}{(k-s+2)^2} \right) \\ \notag
&=  C_\Omega \frac{k! (C_\Omega L)^k \|y\|^{k+1}}{(k+1)^2}  \,\left(\underset{s=2}{\overset{k+1}{\sum}} L^{1-s} \frac{(k+1)^2 20^s}{(k-s+2)^2} \right) \left (1+\frac{1}{L C_\Omega} \right ) \\
&  \leq 2C_\Omega \frac{k! (C_\Omega L)^k \|y\|^{k+1}}{(k+1)^2}  \,\left(\underset{s=2}{\overset{k+1}{\sum}} L^{1-s} \frac{(k+1)^2 20^s}{(k-s+2)^2} \right)  ,
\end{align}
where we used again Lemma \ref{LemmeChemin}  and that $C_\Omega L >1 $ for $L$ large enough. 

\paragraph{Normal trace.}
For the boundary term, using Proposition \ref{P1norm} and \eqref{RhoAnalytique}, we have that 
\begin{align}\label{ineq:boundary}
    \| D^k u \cdot n\|_{\partial \Omega} & \leq \underset{s=2}{\overset{k+1}{\sum}} \| \nabla^s \rho\|  \underset{|\alpha|=k+1-s}{\sum} \frac{k!}{\alpha! (s-1)!} \  \underset{i=1}{\overset{s}{\prod}} \|D^{\alpha_i} u\|\\ \notag
    &\leq \underset{s=2}{\overset{k+1}{\sum}} c_\rho^s s!   \underset{|\alpha|=k+1-s}{\sum} \frac{k!}{\alpha! (s-1)!} \ \underset{i=1}{\overset{s}{\prod}} \frac{\alpha_i! C_\Omega (C_\Omega L)^{\alpha_i} \|y\|^{\alpha_i+1}}{(\alpha_i+1)^2}\\ \notag
    &\leq C_\Omega \frac{k! \|y\|^{k+1}(C_\Omega L)^k}{(k+1)^2} \underset{s=2}{\overset{k+1}{\sum}}  c_\rho^s s L^{1-s} \frac{(k+1)^2 20^s}{(k+2-s)^2} .
\end{align}
\paragraph{Circulations.}
Eventually, using Proposition \ref{P3}, we have that 
\begin{equation} \label{egua}
  \Pi D^k u = \Pi K^k[u,y-u] ,  
\end{equation}
with
\begin{align*}
    |K^k[u,y-u] |& \leq \underset{r=1}{\overset{k-1}{\sum}} |c_{k,r}|\, \|D^{r-1} u\|\, |D^{k-r}(y-u)| + \|D^{k-1}u\| \, |y-u| \\
    &\leq \underset{r=1}{\overset{k-1}{\sum}} \frac{k!}{r!(k-r)!}\frac{(r-1)!}{r^2} C_\Omega (C_\Omega L)^{r-1} \|y\|^r \frac{(k-r)!}{(k-r+1)^2} C_\Omega (C_\Omega L)^{k-r} \|y\|^{k-r+1}\\
    &+\frac{(k-1)! C_\Omega (C_\Omega L)^{k-1}}{k^2}\|y\|^k (C_\Omega +1)|y|,
\end{align*}
    where we used that $D^{k-r} y = 0 $ for $r \leq k-1$, the triangle inequality and 
\eqref{eq:def_C_Omega}. This yields
 \begin{align*}
    |K^k[u,y-u] |& \leq   
 C_\Omega^2 \|y\|^{k+1} (C_\Omega L)^{k-1} k! \underset{r=1}{\overset{k-1}{\sum}}\frac{1}{r^3(k-r+1)^2}+ k! \|y\|^{k+1} C_\Omega (C_\Omega L)^{k-1} (C_\Omega +1) \frac{1}{k^3}\\
   &\leq \frac{C_\Omega +1}{C_\Omega L} C_\Omega \|y\|^{k+1} (C_\Omega L)^{k} k! \underset{r=1}{\overset{k}{\sum}}\frac{1}{r^3(k-r+1)^2} .
\end{align*}
By \eqref{egua} we deduce that 
\begin{equation}\label{ineq_Pi}
|\Pi D^k u|\leq C_\Gamma   \frac{C_\Omega +1}{C_\Omega L} C_\Omega  \|y\|^{k+1} (C_\Omega L)^{k} k! \underset{r=1}{\overset{k}{\sum}}\frac{1}{r^3(k-r+1)^2} .
\end{equation}

\paragraph{Endgame.}
Replacing \eqref{ineq_div}, \eqref{ineq_curl},\eqref{ineq:boundary},\eqref{ineq_Pi} into \eqref{indu2} together with \eqref{def_gamma} we get 
$$
\|D_k u \|\leq c_{\mathfrak r} C_\Omega \frac{k! \|y\|^{k+1} (C_\Omega L)^k }{(k+1)^2} \gamma(L).
$$
We then conclude thanks to \eqref{hyp_L} that \eqref{indu} holds true at order $k$. 
By \eqref{chainrule}, we deduce that 
$\Phi$ is analytic from $(-T, T)$ to $C^{1, \alpha}(\Omega)$ with $T$ depending only on $y_0$ and $\Omega$. 
This ends the proof of Theorem 
\ref{main}.

\section{Remarks and open questions} \label{section7}

\paragraph{More limited regularity.}
Regarding the main result on the analyticity of the trajectories in Theorem \ref{main}, we emphasize the following  possible extensions and open questions.

    \begin{itemize}
        \item Since the result is obtained through a recursive argument, one can also show that if the regularity of the boundary of the domain is $C^k$ then the trajectories are $C^k$ in time.
        \item For less regular boundary domains, following \cite[Theorem 13]{Sueur}, one can show that if the boundary is $C^2$ and $y_0$ constant outside of a compact $K \subset \Omega$, then one can recover the analyticity of the flow map for a small time $T>0$.
        \item For less regular data, with Yudovich type solutions in the case of Lipschitz initial data,  one could expect to get less regularity then analyticity, for instance Gevrey regularity.
    \end{itemize}

\paragraph{Open questions for large times.}

The idea introduced by Angenent, Haker and Tannenbaum  in \cite{AHT} 
is that if the solution to \eqref{AHT} with initial data $y_0$ has a limit at large times, then this limit should be the unique optimal rearrangement of the initial data $y_0$.  Recall that  two $L^2$ maps $y_1, y_2:\Omega\to \R^d$ are rearrangements of each other if they define the same image measure of the Lebesgue measure, i.e.
\[
\int_\Omega f(y_1(x))dx=\int_\Omega f(y_2(x))dx ,
\]
for all compactly supported continuous function $f:\R^d\to \R$. 
The mass transport problem as first formulated by Monge, and reformulated in a modern mathematical formulation by Kantorovich, is then the issue of finding the optimal way, in the sense of minimizing a transportation cost, to rearrange.  
 
In \cite{Brenier91}, Brenier proved that for each $L^2$ map $y_0: \Omega\to \R^d$ there exists a unique rearrangement  $y^*$ with  convex potential, i.e. $y^*=\nabla p^*$ for some convex function $p^*$. Moreover,   $y^*$ minimizes the quadratic cost function 
\[
\int_{\Omega}|y(x)-x|^2dx,
\]
among all possible rearrangements of $y_0$.

Note that the global in time issue remains an open question in general, as explained in \cite{AHT}. However, in addition to the local existence and uniqueness result mentioned above, the authors prove global existence for regularized solutions $\{\gamma_\varepsilon^t, t\geq 0\}$ that are weakly converging to a family of measures $\{\gamma^t, t\geq 0\}$ with decreasing costs and whose $\omega$-limit set consists of critical points of the Monge-Kantorovich functional.
Moreover, they observed that, since the velocity $u$ is divergence-free and tangent to the boundary, we have that for each $t>0$, $y(t)$ is a rearrangement of $y(0)$. 
Moreover, 
$$ \frac{d}{dt} \int_\Omega \frac12 |y-x|^2 \; dx  = - \int_\Omega |u|^2\; dx .$$
Therefore, for any  
steady state $y$, the Leray projection $u=\P y$ vanishes so that $y$ is a gradient field and 
conversely, all gradients are steady states. 
Moreover if a steady state is the gradient of 
 a convex function, then by the uniqueness part of Brenier's theorem \cite{Brenier91}, it has to be then the rearrangement of $y(0)$. 
The idea promoted by 
Angenent, Haker and Tannenbaum is therefore that one may
 obtain some numerical approximations of  the optimal rearrangement of a given $y_0$
 by  considering simulations for large times of the corresponding solution of the AHT.

 In \cite{Nguyen^2} the authors succeeded to prove that strictly convex potentials are stable, while the case of merely convex potential is still open. 
 The question whether the Lipschitz norm, in space, of such solutions may blow up in finite
time as it would be the case of the inviscid Burgers equation, is open. In \cite{Nguyen^2}, Huy Q. Nguyen and Toan T. Nguyen 
proved that for initial maps which are sufficiently close to maps with strictly convex potential, the corresponding solutions to the AHT System are global in time and converge exponentially fast to the optimal rearrangement of the initial map as time tends to infinity.   

On the other hand, in \cite[Theorem 6.1]{DrivasElgindi}, the authors exhibit necessary and sufficient conditions on the initial incompressible velocity $u_0$ for which the Burgers equation admits a global solution on $\mathbb{T}^d$ and $\R^d$. Such a solution is also a global solution to the \eqref{AHT} system with $y=\mathbb{P}y=u$.



\bigskip \ \par \noindent {\bf Acknowledgements.} F. S. was partially supported by the Agence Nationale de la Recherche,  Project BOURGEONS, grant ANR-23-CE40-0014-01. 


\Addresses

\end{document}